\documentclass[12pt]{amsart}
\usepackage{amssymb}
\input amssym.def
\usepackage{pstricks}
\usepackage[dvips]{hyperref}
\usepackage{url}
\newtheorem{theorem}{Theorem}
\newtheorem{lemma}[theorem]{Lemma}
\newtheorem{prop}[theorem]{Proposition}
\newtheorem{cor}[theorem]{Corollary}
\theoremstyle{definition}
\newtheorem{rem}[theorem]{Remark}
\newcommand{\Dif}{\textrm{Homeo}}
\newcommand{\M}{\mathcal{M}}
\newcommand{\T}{\mathcal{T}}
\newcommand{\Y}{\mathcal{Y}}
\newcommand{\calI}{\mathcal{I}}
\newcommand{\Z}{\mathbb{Z}}
\newcommand{\R}{\mathbb{R}}
\newcommand{\bdr}[1]{\partial\! #1}
\newcommand{\lr}[1]{\left<#1\right>}
\newcommand{\Sp}{\mathrm{Sp}}
\newcommand{\Stab}{\mathrm{Stab}}
\newcommand{\Iso}{\mathrm{Iso}}
\numberwithin{equation}{section}
\numberwithin{theorem}{section}
\author{B\l{}a\.zej Szepietowski}
\title[Crosscap slides and level 2 mapping class group]{Crosscap slides and the level 2  mapping class group of a nonorientable surface}
\address[]{Institute of Mathematics, Gda\'nsk University, Wita Stwosza 57,
80-952 Gda\'nsk, Poland} 

\email{blaszep@mat.ug.edu.pl}
\thanks{Supported by the MNiSW grant N N201 366436}

\begin{document}
\begin{abstract}
Crosscap slide is a homeomorphism of a nonorientable surface of genus at least 2, which was introduced under the name  Y-homeomorphism by Lickorish as an example of an element of the mapping class group which cannot be expressed as a product of Dehn twists. We prove that the subgroup of the mapping class group of  a closed nonorientable surface $N$ generated by all crosscap slides is equal to the level 2 subgroup consisting of those mapping classes which act trivially on  $H_1(N;\Z_2)$. We also prove that this subgroup is generated by involutions.
\end{abstract}

\maketitle

\section{Introduction}
It is well known that the mapping class group $\M(S)$ of a compact orientable surface $S$ is generated by Dehn twists. This result was originally proved by Dehn \cite{Dehn} and rediscovered  by Lickorish \cite{Lick0}. On a compact nonorientable surface $N$, however, there are elements of the mapping class group $\M(N)$ which cannot be expressed as a product of Dehn twists. The first example of such element is the crosscap slide, also called Y-homeomorphism, introduced by Lickorish \cite{Lick1}. He proved \cite{Lick1,Lick2} that if $N$ is a closed nonorientable surface of genus at least 2, then the twist subgroup $\T(N)$, generated by Dehn twists about all two-sided simple closed curves, has index 2 in $\M(N)$, and that $\M(N)$ is generated by Dehn twists and one crosscap slide. Later, Chillingworth \cite{Chill} found finite generating sets for $\M(N)$ and $\T(N)$. Korkmaz \cite{KorkHom,Kork} and Stukow \cite{Stu2,Stu3} extended Chillingworth's results to surfaces with punctures and boundary and also computed the abelianizations of $\M(N)$ and $\T(N)$. 

In this paper we consider the subgroup $\Y(N)$ of $\M(N)$ generated by all crosscap slides. 
Our main result is Theorem \ref{main}, which asserts that for a closed nonorientable surface $N$ of genus $g\ge 2$, $\Y(N)$ is equal to the level 2 subgroup $\Gamma_2(N)$ of $\M(N)$ consisting of those mapping classes which act trivially on $H_1(N;\Z_2)$. In particular $\Y(N)$ is a subgroup of finite index. We also prove that $\Y(N)$ is generated by involutions (Theorem \ref{invol}). 

It is an interesting problem to find a finite generating set for $\Y(N)$ and also compute its abelianization. For an orientable surface $S$, the abelianization of the level 2 subgroup of $\M(S)$  was computed by Sato \cite{Sato}.

The paper is organised as follows. After Preliminaries, we prove in Section \ref{s3} that certain elements of $\M(N)$ belong to $\Y(N)$ and that the last group is generated by involutions. In Section \ref{s4} we obtain convenient generating sets for the level $2$ mapping class groups of orientable surfaces with boundary. In Section \ref{s5} we use lemmas proved in previous sections, as well as some results of McCarthy-Pinkall \cite{McCP} and
Korkmaz \cite{KorkHom}, to prove our main result, which identifies the group $\Y(N)$ with the level 2 mapping class group $\Gamma_2(N)$. 

\medskip

\noindent{\bf Acknowledgment.} I was informed by Susumu Hirose that he found a shorter proof of my Theorem \ref{main} using Nowik's result \cite[Theorem 3.6]{Now}. I thank him for the discussion and for the reference \cite{Now}.
I also wish to thank the referee for helpful comments.

\section{Preliminaries}
Let $F$ be a connected compact surface, possibly with
boundary. Define $\Dif(F)$ to be the group of all, orientation
preserving if $F$ is orientable, homeomorphisms $h\colon F\to F$
equal to the identity on the boundary $\bdr{F}$. The {\it mapping class
group} $\M(F)$ is the group of isotopy classes in $\Dif(F)$. By
abuse of notation we will use the same symbol to denote a
homeomorphism and its isotopy class. If $f$ and $h$ are two
homeomorphisms, then the composition $fh$ means that $h$ is applied
first. 
A surface of genus $g$ with $n$ boundary components will be denoted by $S_{g,n}$ if it is orientable, or by $N_{g,n}$ if it is nonorientable. The genus of a nonorientable surface is the number of projective planes in the connected sum decomposition.

\subsection{Simple closed curves and Dehn twists.}
By a {\it simple closed curve} in $F$ we mean an embedding
$\gamma\colon S^1\to F\backslash\bdr{F}$. Note that $\gamma$ has an orientation; the
curve with the opposite orientation but same image will be denoted by
$\gamma^{-1}$. 
By abuse of notation, we will often identify a simple closed curve with its
oriented image and also with its isotopy class. 
According to whether a regular neighborhood of $\gamma$ is an annulus or a M\"obius strip, we call $\gamma$ respectively {\it two-} or {\it one-sided}. 
We say that $\gamma$ is {\it nonseparating} if $F\backslash\gamma$
is connected and {\it separating} otherwise. If $\gamma=\bdr{M}$ for some subsurface $M\subset F$ then we call $\gamma$ {\it bounding}. If $n\le 1$ then every separating curve is bounding.

Given a two-sided simple closed curve $\gamma$, $T_\gamma$ denotes a Dehn
twist about $\gamma$. On a  nonorientable surface it is
impossible to distinguish between right and left twists, so the
direction of a twist $T_\gamma$ has to be specified for each curve
$\gamma$. Equivalently we may choose an orientation of a regular
neighborhood of $\gamma$. Then $T_\gamma$ denotes the right Dehn twist with
respect to the chosen orientation. Unless we specify which of the
two twists we mean, $T_\gamma$ denotes any of the
two possible twists. Recall that $T_\gamma$ does not depend on the orientation of $\gamma$.

{\it Bounding pair} is a pair $(\gamma,\gamma')$ of disjoint 
simple closed curves, such that $\gamma\cup\gamma'=\bdr{S}$ for an orientable subsurface $S\subset F$. If $T_{\gamma}$, $T_{\gamma'}$ are right Dehn twists with respect to some orientation of $S$, then $T_{\gamma}T^{-1}_{\gamma'}$ is called {\it bounding pair map}. Our definition of a bounding pair differs slightly from the usual definition, in which the curves $\gamma$ and $\gamma'$ are required to be nonseparating.

\begin{figure}
\input{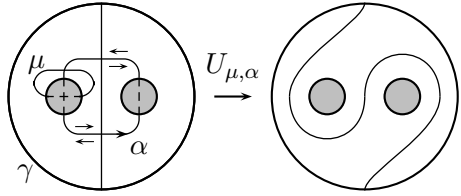}
\caption{\label{U} Crosscap transposition.}
\end{figure}

\begin{figure}
\input{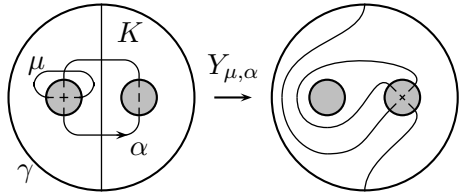}
\caption{\label{Y} Crosscap slide.}
\end{figure}

\subsection{Crosscap slides.}
We begin this subsection by describing a convention used in all figures in this paper. We explain this on the example of Figure \ref{U}. The shaded discs
represent crosscaps; this means that their interiors should be
removed, and then antipodal points in each resulting boundary
component should be identified. The small arrows on two sides of the curve   $\alpha$ indicate the direction of the Dehn twist $T_\alpha$.

Let $N=N_{g,n}$ be a nonorientable surface of genus $g\ge 2$.
Suppose that $\mu$ and $\alpha$ are two simple closed curves in $N$,
such that $\mu$ is one-sided, $\alpha$ is two-sided and they
intersect in one point. Let $K\subset N$ be a regular neighborhood of
$\mu\cup\alpha$, 
which is homeomorphic to the Klein bottle with a
hole. On Figure \ref{U} a homeomorphism of $K$ is shown, which interchanges the two crosscaps keeping the boundary of $K$ fixed. It may be extended 
by the identity outside $K$ to a homeomorphism of $N$, which
we call {\it crosscap
transposition} and denote as $U_{\mu,\alpha}$.
We define {\it crosscap slide} $Y_{\mu,\alpha}$ to be the composition
\[Y_{\mu,\alpha}=U_{\mu,\alpha}T_\alpha,\]
where $T_\alpha$ is the Dehn twist about $\alpha$ in the direction indicated by the arrows in Figure \ref{U}. If $M\subset K$ is a regular neighbourhood of $\mu$, which is a M\"obiu strip, then $Y_{\mu,\alpha}$ may be described as the effect of pushing $M$ once along $\alpha$ (Figure \ref{Y}).
Observe that $Y_{\mu,\alpha}$ reverses the orientation of $\mu$. 
 Up to isotopy, $Y_{\mu,\alpha}$ does not depend on the
choice of the regular neighbourhood $K$. It also does not depend on the orientation of $\mu$
but does depend on the orientation of $\alpha$. The following
properties of crosscap slides are easy to verify.
\begin{equation}\label{y^-1}
Y_{\mu,\alpha^{-1}}=Y^{-1}_{\mu,\alpha},
\end{equation}
\begin{equation}
Y_{\mu,\alpha}T_\alpha Y_{\mu,\alpha}^{-1}=U_{\mu,\alpha}T_\alpha U_{\mu,\alpha}^{-1}=T^{-1}_\alpha,
\end{equation}
\begin{equation}\label{y2}
Y_{\mu,\alpha}^2=U_{\mu,\alpha}^2=T_\gamma,
\end{equation}
where $T_\gamma$ is the Dehn twist about $\gamma=\bdr{K}$, right with respect
to the standard orientation of the plane of Figure \ref{Y},
\begin{equation}\label{hyh}
hY_{\mu,\alpha}h^{-1}=Y_{h(\mu),h(\alpha)},
\end{equation}
for all $h\in\Dif(N)$. 
The crosscap slide is an example of an element of $\M(N)$ which is not a product of Dehn twists. It was introduced under the name Y-homeomorphism by Lickorish, who proved that $\M(N_g)$ is generated by Dehn twists and one crosscap slide for $g\ge 2$ \cite{Lick1,Lick2}. 

\subsection{Blowup homomorphism and crosscap pushing map.}\label{blowup}
Let $N=N_{g,n}$ and fix $x_0\in N\backslash\bdr{N}$.
Define $\M(N,x_0)$ to be the group of isotopy classes in $\Dif(N,x_0)$, the group of all homeomorphisms $h\colon N\to N$ equal to the identity on $\bdr{N}$ and such that $h(x_0)=x_0$. 

Let $U=\{z\in\mathbb{C}\,|\,|z|\le 1\}$ and fix an embedding
$e\colon U\to N\backslash\bdr{N}$ such that $e(0)=x_0$.
The surface $N_{g+1,n}$ may be obtained by removing from $N$ the interior of $e(U)$ and then identifying $e(z)$ with $e(-z)$ for $z\in S^1=\bdr{U}$.
Let $\gamma=e(S^1)$ and let $\mu$ be the image of $\gamma$ in $N_{g+1,n}$.
Note that $\mu$ is a one-sided simple closed curve.

We define a {\it blowup homomorphism} 
\[\varphi\colon\M(N_{g,n},x_0)\to\M(N_{g+1,n})\] as follows. 
Represent $h\in\M(N,x_0)$ by a homeomorphism $h\colon N\to N$ such that (a) $h$ is equal to the identity on $e(U)$ or (b) $h(x)=e(\overline{e^{-1}(x)})$ for $x\in e(U)$. Such $h$ commutes with the identification leading to $N_{g+1,n}$ and thus induces an element $\varphi(h)\in\M(N_{g+1,n})$. To see that $\varphi$ is well defined consider two homeomorphisms $h_1,h_2\in\Dif(N,x_0)$ which satisfy (a) or (b) and are isotopic by an isotopy fixing $x_0$. Let $N'=N\backslash int(e(U))$. By \cite[Theorem 3.6]{Stu4} the kernel of the homomorphism $i_\ast\colon\M(N')\to\M(N,x_0)$ induced by the inclusion of $N'$ in $N$ is generated by $T_\gamma$. It follows that the restriction of $h_1h^{-1}_2$ to $N'$ is isotopic by an isotopy fixed on $\gamma$ to some power of $T_\gamma$. Since a Dehn twist about a curve bounding a M\"obius strip is trivial, $h_1$ and $h_2$ induce the same element of $\M(N_{g+1,n})$.

Forgetting the distinguished point $x_0$ induces a homomorphism
\[\M(N_{g,n},x_0)\to\M(N_{g,n}),\] which fits into the Birman exact sequence (see \cite{Kork})
\[\pi_1(N_{g,n},x_0)\stackrel{j}{\to}\M(N_{g,n},x_0)\to\M(N_{g,n})\to 1.\]
The homomorphism $j$ is called {\it point pushing map}.
If $\alpha$ is a loop in $N_{g,n}$ based at $x_0$ and $[\alpha]\in\pi_1(N_{g,n},x_0)$ is  its homotopy class, then $j([\alpha])$ may be described as the effect of pushing $x_0$ once along $\alpha$. 

We define a {\it crosscap pushing map} to be the composition 
\[\psi=\varphi\circ j\colon\pi_1(N_{g,n},x_0)\to\M(N_{g+1,n}).\] 
The following two lemmas follow from the description of the point pushing map for nonorientable surfaces \cite[Lemma 2.2 and Lemma 2.3]{Kork} and the definition of a crosscap slide.

\begin{lemma}\label{push2}
Suppose that $\alpha$ is a two-sided simple loop in $N_{g,n}$ based at $x_0$ and let $M$ be a regular neighbourhood of $\alpha$ containing $e(U)$. Let $M'$ be the image in $N_{g+1,n}$ of $M\backslash int(e(U))$.
Then \[\psi([\alpha])=T_\beta T_\delta^{-1},\]
where $\beta$ and $\delta$ are boundary curves of $M'$ and the twists are right with respect to some orientation of $M'\backslash\mu$.
\end{lemma}

\begin{lemma}\label{push1}
Suppose that $\alpha$ is a one-sided simple  loop in $N_{g,n}$ based at $x_0$.
Let $\alpha'$ be a simple loop in the homotopy class $[\alpha]$ which
intersects $\gamma$ in two antipodal points.
Then \[\psi([\alpha])=Y_{\mu,\tilde{\alpha}},\]
where $\tilde{\alpha}$ is the image in $N_{g+1,n}$ of $\alpha'\backslash int(e(U))$.
\end{lemma}

\begin{cor}\label{pushinY}
For $n\le 1$ every element of $\psi(\pi_1(N_{g,n},x_0))$ is a product of crosscap slides.
\end{cor}
\begin{proof}
For $n\le 1$,
$\pi_1(N_{g,n},x_0)$ is generated by homotopy classes of one-sided simple loops.
Since $\psi$ is a homomorphism, it follows by Lemma \ref{push1} that $\psi(\pi_1(N_{g,n},x_0))$ is generated by crosscap slides.
\end{proof}

\section{The group generated by crosscap slides.}\label{s3}
In this section we assume that $N=N_g$ is a closed nonorientable surface of genus $g\ge 2$. We denote by $\Y(N)$  the subgroup of $\M(N)$ generated by all crosscap slides.
Since by (\ref{hyh}) a conjugate of a crosscap slide is also a crosscap slide, $\Y(N)$ is a normal subgroup of $\M(N)$.

\begin{lemma}\label{square}
For every nonseparating two-sided simple closed curve $\alpha$ in $N$, $T^2_\alpha$ belongs to $\Y(N)$.
\end{lemma}

\begin{proof} Let $\mu$ be a one-sided curve intersecting $\alpha$ in one point. 
Since $Y_{\mu,\alpha}$ preserves the curve $\alpha$ and reverses orientation of its neighbourhood, we have $Y_{\mu,\alpha}T_\alpha Y_{\mu,\alpha}^{-1}=T^{-1}_\alpha,$
and by (\ref{y^-1}, \ref{hyh}) 
\[T^2_\alpha=T_\alpha Y_{\mu,\alpha}T^{-1}_\alpha Y^{-1}_{\mu,\alpha}
=Y_{T_\alpha(\mu),\alpha}Y_{\mu,\alpha^{-1}}.\]
\end{proof}

\begin{figure}\label{lantern}
\input{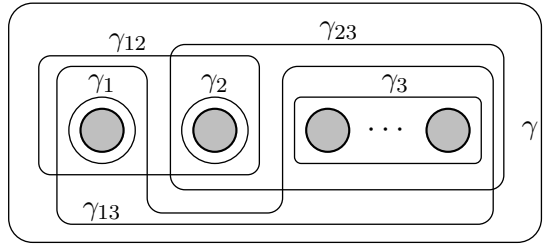}
\caption{Curves of the lantern relation.}
\end{figure}

\begin{lemma}\label{separating}
For every separating simple closed curve $\gamma$ in $N$, $T_\gamma$ belongs to $\Y(N)$.
\end{lemma}

\begin{proof} Let $K$ be a nonorientable subsurface of $N$ such that $\gamma=\bdr{K}$. We will argue by induction on the genus $\rho$ of $K$. If $\rho=1$ then $T_\gamma$ is trivial, and if $\rho=2$ then $T_\gamma$ is the square of a crosscap slide supported in $K$ by (\ref{y2}). Suppose that $\rho\ge 3$ and consider the
separating curves in Figure \ref{lantern}. Note that $\gamma, \gamma_1, \gamma_2, \gamma_3, $ bound a four-holed sphere and we have the well known lantern relation
\[T_\gamma T_{\gamma_1}T_{\gamma_2}T_{\gamma_3}=T_{\gamma_{12}}T_{\gamma_{13}}T_{\gamma_{23}},\]
where the twists are right with respect to the standard orientation of the plane of the figure.
Since each of the curves different from $\gamma$ bounds a nonorientable subsurface of genus less then $\rho$, the corresponding twist is in $\Y(N)$ by inductive hypothesis. It follows that $T_\gamma\in\Y(N)$.\end{proof}

\begin{figure}
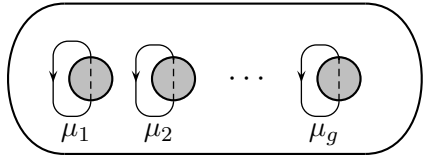

\pspicture*(6,2.5)
\psline(1,.25)(5,.25) \psline(1,2.25)(5,2.25)
\psbezier(1,.25)(0,.25)(0,2.25)(1,2.25)
\psbezier(5,.25)(6,.25)(6,2.25)(5,2.25)
\pscircle*[linecolor=lightgray](1.35,1.25){.3}
\pscircle(1.35,1.25){.3}
\pscircle*[linecolor=lightgray](2.45,1.25){.3}
\pscircle(2.45,1.25){.3}
\rput[c](3.45,1.25){$\cdots$}
%
\pscircle*[linecolor=lightgray](4.65,1.25){.3}
\pscircle(4.65,1.25){.3}
\rput[t](1.15,.65){\small$\mu_1$}
\psline[linewidth=.5pt, arrowsize=2pt 2.5]{->}(.85,1.3)(.85,1.2)
\psline[linewidth=.5pt,linearc=.2](1.35,1.55)(1.35,1.75)(.85,1.75)(.85,.75)(1.35,.75)(1.35,.95)
\psline[linewidth=.5pt,linestyle=dashed,dash=3pt
2pt](1.35,.95)(1.35,1.55)
\rput[t](2.25,.65){\small$\mu_2$}
\psline[linewidth=.5pt, arrowsize=2pt 2.5]{->}(1.95,1.3)(1.95,1.2)
\psline[linewidth=.5pt,linearc=.2](2.45,1.55)(2.45,1.7)(1.95,1.7)(1.95,.75)(2.45,.75)(2.45,.95)
\psline[linewidth=.5pt,linestyle=dashed,dash=3pt
2pt](2.45,.95)(2.45,1.55)
\psline[linewidth=.5pt, arrowsize=2pt 2.5]{->}(4.15,1.3)(4.15,1.2)
\rput[t](4.45,.65){\small$\mu_g$}
\psline[linewidth=.5pt,linearc=.2](4.65,1.55)(4.65,1.7)(4.15,1.7)(4.15,.75)(4.65,.75)(4.65,.95)
\psline[linewidth=.5pt,linestyle=dashed,dash=3pt
2pt](4.65,.95)(4.65,1.55)
%

%
\endpspicture
\caption{\label{mui}The one-sided curves $\mu_i$.}
\end{figure}

\begin{figure}
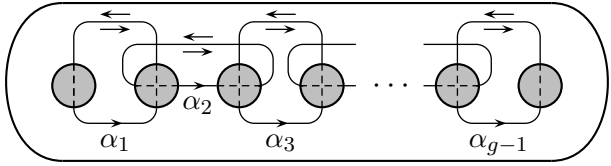

\pspicture*(8.5,2.5)
\psline(1,.25)(7.5,.25) \psline(1,2.35)(7.5,2.35)
\psbezier(1,.25)(0,.25)(0,2.35)(1,2.35)
\psbezier(7.5,.25)(8.5,.25)(8.5,2.35)(7.5,2.35)
\pscircle*[linecolor=lightgray](1.15,1.25){.3}
\pscircle(1.15,1.25){.3}
\pscircle*[linecolor=lightgray](2.25,1.25){.3}
\pscircle(2.25,1.25){.3}
\pscircle*[linecolor=lightgray](3.35,1.25){.3}
\pscircle(3.35,1.25){.3}
\pscircle*[linecolor=lightgray](4.45,1.25){.3}
\pscircle(4.45,1.25){.3}
\pscircle*[linecolor=lightgray](6.25,1.25){.3}
\pscircle(6.25,1.25){.3}
\pscircle*[linecolor=lightgray](7.35,1.25){.3}
\pscircle(7.35,1.25){.3}
\psline[linewidth=.5pt, arrowsize=2pt 2.5]{->}(1.7,.75)(1.8,.75)
\rput[t](1.7,.6){\small$\alpha_1$}
\psline[linewidth=.5pt,linearc=.2](1.15,1.55)(1.15,2.1)(2.25,2.1)(2.25,1.55)
\psline[linewidth=.5pt,linearc=.2](1.15,.95)(1.15,.75)(2.25,.75)(2.25,.95)
\psline[linewidth=.5pt,linestyle=dashed,dash=3pt
2pt](1.15,.95)(1.15,1.55)
\psline[linewidth=.5pt, arrowsize=2pt 2.5]{<-}(1.5,2.2)(1.9,2.2)
\psline[linewidth=.5pt, arrowsize=2pt 2.5]{->}(1.5,2)(1.9,2)
\psline[linewidth=.5pt, arrowsize=2pt 2.5]{->}(3.9,.75)(4,.75)
\rput[t](3.9,.6){\small$\alpha_3$}
\psline[linewidth=.5pt,linearc=.2](3.35,1.55)(3.35,2.1)(4.45,2.1)(4.45,1.55)
\psline[linewidth=.5pt,linearc=.2](3.35,.95)(3.35,.75)(4.45,.75)(4.45,.95)
\psline[linewidth=.5pt,linestyle=dashed,dash=3pt
2pt](1.15,.95)(1.15,1.55)
\psline[linewidth=.5pt,linestyle=dashed,dash=3pt
2pt](2.25,.95)(2.25,1.55)
\psline[linewidth=.5pt,linestyle=dashed,dash=3pt
2pt](3.35,.95)(3.35,1.55)
\psline[linewidth=.5pt, arrowsize=2pt 2.5]{<-}(3.7,2.2)(4.1,2.2)
\psline[linewidth=.5pt, arrowsize=2pt 2.5]{->}(3.7,2)(4.1,2)
\psline[linewidth=.5pt, arrowsize=2pt 2.5]{->}(2.8,1.25)(2.9,1.25)
\rput[t](2.8,1.1){\small$\alpha_2$}
\psline[linewidth=.5pt,linearc=.2](1.95,1.25)(1.8,1.25)(1.8,1.8)(3.8,1.8)(3.8,1.25)(3.65,1.25)
\psline[linewidth=.5pt,linestyle=dashed,dash=3pt
2pt](1.95,1.25)(2.55,1.25)
\psline[linewidth=.5pt](2.55,1.25)(3.05,1.25)
\psline[linewidth=.5pt,linestyle=dashed,dash=3pt
2pt](3.05,1.25)(3.65,1.25)
\psline[linewidth=.5pt, arrowsize=2pt 2.5]{<-}(2.6,1.9)(3,1.9)
\psline[linewidth=.5pt, arrowsize=2pt 2.5]{->}(2.6,1.7)(3,1.7)
\psline[linewidth=.5pt,linestyle=dashed,dash=3pt
2pt](4.45,.95)(4.45,1.55)
\psline[linewidth=.5pt, arrowsize=2pt 2.5]{->}(6.8,.75)(6.9,.75)
\rput[t](6.8,.6){\small$\alpha_{g-1}$}
\psline[linewidth=.5pt,linearc=.2](6.25,1.55)(6.25,2.1)(7.35,2.1)(7.35,1.55)
\psline[linewidth=.5pt,linearc=.2](6.25,.95)(6.25,.75)(7.35,.75)(7.35,.95)
\psline[linewidth=.5pt,linestyle=dashed,dash=3pt
2pt](6.25,.95)(6.25,1.55)
\psline[linewidth=.5pt,linestyle=dashed,dash=3pt
2pt](7.35,.95)(7.35,1.55)
\psline[linewidth=.5pt, arrowsize=2pt 2.5]{<-}(6.6,2.2)(7,2.2)
\psline[linewidth=.5pt, arrowsize=2pt 2.5]{->}(6.6,2)(7,2)
\psline[linewidth=.5pt,linearc=.2](4.15,1.25)(4,1.25)(4,1.8)(4.9,1.8)
\psline[linewidth=.5pt,linestyle=dashed,dash=3pt
2pt](4.15,1.25)(4.75,1.25)
\psline[linewidth=.5pt](4.75,1.25)(4.9,1.25)
\rput[c](5.4,1.25){$\cdots$}
\psline[linewidth=.5pt,linearc=.2](6.55,1.25)(6.7,1.25)(6.7,1.8)(5.8,1.8)
\psline[linewidth=.5pt,linestyle=dashed,dash=3pt
2pt](5.95,1.25)(6.55,1.25)
\psline[linewidth=.5pt](5.8,1.25)(5.95,1.25)
%
\endpspicture
\caption{\label{alphai}The two-sided curves $\alpha_i$.}
\end{figure}

\begin{figure}
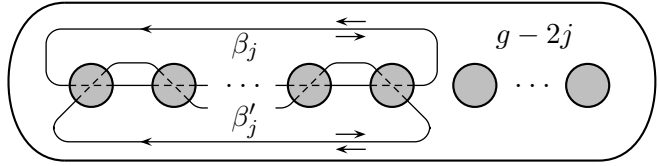

\pspicture*(9.1,2.5)
\psline(1,.25)(8.1,.25) \psline(1,2.35)(8.1,2.35)
\psbezier(1,.25)(0,.25)(0,2.35)(1,2.35)
\psbezier(8.1,.25)(9.1,.25)(9.1,2.35)(8.1,2.35) %
\pscircle*[linecolor=lightgray](1.35,1.25){.3}
\pscircle(1.35,1.25){.3}
\pscircle*[linecolor=lightgray](2.45,1.25){.3}
\pscircle(2.45,1.25){.3}
\pscircle*[linecolor=lightgray](4.25,1.25){.3}
\pscircle(4.25,1.25){.3}
\pscircle*[linecolor=lightgray](5.35,1.25){.3}
\pscircle(5.35,1.25){.3}
\pscircle*[linecolor=lightgray](6.45,1.25){.3}
\pscircle(6.45,1.25){.3}
\rput[c](7.25,1.25){$\dots$}
\rput[c](7.25,1.9){\small{$g-2j$}}	
\pscircle*[linecolor=lightgray](7.95,1.25){.3}
\pscircle(7.95,1.25){.3}
\psline[linewidth=.5pt, arrowsize=2pt 2.5]{->}(4.6,.6)(5,.6)
\psline[linewidth=.5pt, arrowsize=2pt 2.5]{<-}(4.6,.4)(5,.4)
\psline[linewidth=.5pt, arrowsize=2pt 2.5]{->}(2.1,.5)(2,.5)
\rput[b](3.4,.55){\small{$\beta'_j$}}
\psline[linewidth=.5pt,linestyle=dashed,dash=3pt
2pt](1.55,1.45)(1.15,1.05)
\psline[linewidth=.5pt,linearc=.2](1.55,1.45)(1.65,1.55)(2.15,1.55)(2.25,1.45)
\psline[linewidth=.5pt,linestyle=dashed,dash=3pt
2pt](2.25,1.45)(2.65,1.05)
\psline[linewidth=.5pt,linearc=.2](2.65,1.05)(2.75,.95)(2.9,.95)
\psline[linewidth=.5pt,linearc=.2](4.05,1.05)(3.95,.95)(3.8,.95)
\psline[linewidth=.5pt,linestyle=dashed,dash=3pt
2pt](4.05,1.05)(4.45,1.45)
\psline[linewidth=.5pt,linearc=.2](4.45,1.45)(4.55,1.55)(5.05,1.55)(5.15,1.45)
\psline[linewidth=.5pt,linestyle=dashed,dash=3pt
2pt](5.55,1.05)(5.15,1.45)
\psline[linewidth=.5pt,linearc=.2](1.15,1.05)(.85,.75)(.85,.5)(5.85,.5)(5.85,.75)(5.55,1.05)
\rput[c](3.4,1.25){$\dots$}
\psline[linewidth=.5pt, arrowsize=2pt 2.5]{<-}(4.6,2.1)(5,2.1)
\psline[linewidth=.5pt, arrowsize=2pt 2.5]{->}(4.6,1.9)(5,1.9)
\psline[linewidth=.5pt, arrowsize=2pt 2.5]{->}(2.1,2)(2,2)
\rput[t](3.4,1.95){\small{$\beta_j$}}
\psline[linewidth=.5pt,linestyle=dashed,dash=3pt
2pt](1.05,1.25)(1.65,1.25)
\psline[linewidth=.5pt](1.65,1.25)(2.15,1.25)
\psline[linewidth=.5pt,linestyle=dashed,dash=3pt
2pt](2.15,1.25)(2.75,1.25)
\psline[linewidth=.5pt](2.75,1.25)(2.9,1.25)
\psline[linewidth=.5pt](3.8,1.25)(3.95,1.25)
\psline[linewidth=.5pt,linestyle=dashed,dash=3pt
2pt](3.95,1.25)(4.55,1.25)
\psline[linewidth=.5pt](4.55,1.25)(5.05,1.25)
\psline[linewidth=.5pt,linestyle=dashed,dash=3pt
2pt](5.05,1.25)(5.65,1.25)
\psline[linewidth=.5pt,linearc=.2](1.05,1.25)(.75,1.25)(.75,2)(5.95,2)(5.95,1.25)(5.65,1.25)
%
%

%
\endpspicture
\caption{\label{betaj}The pair $(\beta_j, \beta'_j)$ bounding an orientable subsurface of genus $j-1$.}
\end{figure}

Represent the surface $N=N_g$ ($g\ge 2$) as a 2-sphere with $g$ crosscaps (Figure \ref{mui}). Recall that it means that the interiors of the shaded disc should be removed and then antipodal points on each resulting boundary component should be identified. Assume that the sphere is embedded in $\R^3$ is such a way that it is invariant under the reflection about a plane which contains centers of 
the shaded discs. The reflection commutes with the identification leading to $N$, so it induces a homeomorphism  $R\colon N\to N$ of order 2.

For $1\le i\le g-1$ let 
\[U_i=U_{\mu_i,\alpha_i},\qquad Y_i=Y_{\mu_i,\alpha_i}=U_{\mu_i,\alpha_i}T_{\alpha_i},\]
where $\mu_i$ is the one sided curve in Figure \ref{mui} and  
$T_{\alpha_i}$ is the Dehn twist about $\alpha_i$ in Figure \ref{alphai} in the indicated direction.  For $1\le j\le\lfloor g/2\rfloor$ let $T_{\beta_j}$, $T_{\beta'_j}$ be Dehn twists about the curves in Figure \ref{betaj} in the indicated directions. Observe that $\beta_j$ and $\beta'_j$ bound an orientable subsurface of genus $j-1$ and $T_{\beta_j}$, $T_{\beta'_j}$ are right twists with respect to  appropriate orientation of this subsurface, hence $T_{\beta_j}T^{-1}_{\beta'_j}$ is a bounding pair map.   

\begin{theorem}[Chillingworth \cite{Chill}]{\label{Ch}}
The group $\M(N_g)$ is generated by
\[\{T_{\alpha_i}, T_{\beta_j}, Y_{g-1}; 1\le i\le g-1, 1\le j\le\lfloor g/2\rfloor\}.\]
\end{theorem}

\begin{lemma}\label{RinY}
The involution $R$ belongs to $\Y(N)$.
\end{lemma}

\begin{proof} Let $F=RU_{g-1}U_{g-2}\cdots U_1T_{\alpha_1}T_{\alpha_2}\cdots T_{\alpha_{g-1}}$. It is straightforward it check that $F$ fixes each of the curves $\mu_i$, $\alpha_i$, $\beta_j$ and preserves orientation of regular neighbourhoods of the last two for $1\le i\le g-1$, $1\le j\le\lfloor g/2\rfloor$. It follows that $F$ commutes
with  $Y_i$, $T_{\alpha_i}$, $T_{\beta_j}$, and thus by Theorem \ref{Ch} it belongs to the center of $\M(N_g)$ which is trivial, according to \cite[Corollary 6.3]{Stu1}. Hence $F=1$ and
\begin{align*}
R&=U_{g-1}U_{g-2}\cdots U_1T_{\alpha_1}T_{\alpha_2}\cdots T_{\alpha_{g-1}}\\
&=U_{g-1}T_{\alpha_{g-1}}T_{\alpha_{g-1}}^{-1}U_{g-2}
T_{\alpha_{g-2}}T_{\alpha_{g-1}}T_{\alpha_{g-1}}^{-1}T_{\alpha_{g-2}}^{-1}\cdots U_1T_{\alpha_1}T_{\alpha_2}\cdots T_{\alpha_{g-1}}\\
&=Y_{g-1}T^{-1}_{\alpha_{g-1}}Y_{g-2}T_{\alpha_{g-1}}\cdots
(T_{\alpha_2}\cdots T_{\alpha_{g-1}})^{-1}Y_1(T_{\alpha_2}\cdots T_{\alpha_{g-1}})\in\Y(N).
\end{align*}
\end{proof}

Let $N_{g-1}$ be the surface obtained by replacing the leftmost crosscap in Figure \ref{mui} by a disc with center $x_0$. Taking that disc as $e(U)$, $N$ may be seen as being obtained from $N_{g-1}$ by the blowup construction described in Subsection \ref{blowup}. As a result we have the crosscap pushing map
\[\psi\colon\pi_1(N_{g-1},x_0)\to\M(N).\]
By Corollary \ref{pushinY}  $\psi(\pi_1(N_{g-1},x_0))\subset\Y(N)$ .

\begin{figure}
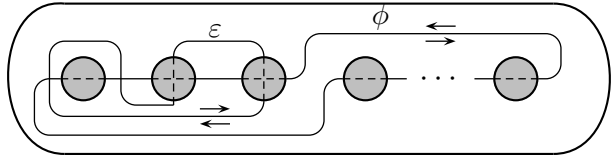

\pspicture*(8.5,2.5)
\psline(1,.25)(7.5,.25) \psline(1,2.25)(7.5,2.25)
\psbezier(1,.25)(0,.25)(0,2.35)(1,2.25)
\psbezier(7.5,.25)(8.5,.25)(8.5,2.25)(7.5,2.25)
\pscircle*[linecolor=lightgray](1.25,1.25){.3}
\pscircle(1.25,1.25){.3}
\pscircle*[linecolor=lightgray](2.45,1.25){.3}
\pscircle(2.45,1.25){.3}
\pscircle*[linecolor=lightgray](3.65,1.25){.3}
\pscircle(3.65,1.25){.3}
\pscircle*[linecolor=lightgray](5,1.25){.3}
\pscircle(5,1.25){.3}
\pscircle*[linecolor=lightgray](7,1.25){.3}
\pscircle(7,1.25){.3}
\rput[b](3,1.8){\small$\varepsilon$}
\psline[linewidth=.5pt,linearc=.2](3.65,.95)(3.65,.75)(.8,.75)(.8,1.75)(1.75,1.75)(1.75,.9)(2.45,.9)
\psline[linewidth=.5pt,linestyle=dashed,dash=3pt 2pt](2.45,.9)(2.45,1.55)
\psline[linewidth=.5pt,linestyle=dashed,dash=3pt 2pt](3.65,.95)(3.65,1.55)
\psline[linewidth=.5pt,linearc=.2](2.45,1.55)(2.45,1.75)(3.65,1.75)(3.65,1.55)
\rput[b](5.2,1.9){\small$\phi$}
\psline[linewidth=.5pt,linearc=.2](.95,1.25)(.6,1.25)(.6,.5)(4.45,.5)(4.45,1.25)(4.7,1.25)
\psline[linewidth=.5pt,linearc=.2](3.95,1.25)(4.2,1.25)(4.2,1.85)(7.6,1.85)(7.6,1.25)(7.3,1.25)
\psline[linewidth=.5pt,linestyle=dashed,dash=3pt 2pt](.95,1.25)(1.55,1.25)
\psline[linewidth=.5pt,linearc=.2](1.55,1.25)(2.15,1.25)
\psline[linewidth=.5pt,linestyle=dashed,dash=3pt 2pt](2.15,1.25)(2.75,1.25)
\psline[linewidth=.5pt,linearc=.2](2.75,1.25)(3.35,1.25)
\psline[linewidth=.5pt,linestyle=dashed,dash=3pt 2pt](3.35,1.25)(3.95,1.25)
\psline[linewidth=.5pt,linestyle=dashed,dash=3pt 2pt](4.7,1.25)(5.3,1.25)
\psline[linewidth=.5pt,linearc=.2](5.3,1.25)(5.55,1.25)
\rput[c](6,1.25){$\cdots$}
\psline[linewidth=.5pt,linearc=.2](6.45,1.25)(6.7,1.25)
\psline[linewidth=.5pt,linestyle=dashed,dash=3pt 2pt](6.7,1.25)(7.3,1.25)
%
%
%
\psline[linewidth=.5pt, arrowsize=2pt 2.5]{<-}(5.8,1.95)(6.2,1.95)
\psline[linewidth=.5pt, arrowsize=2pt 2.5]{->}(5.8,1.75)(6.2,1.75)
\psline[linewidth=.5pt, arrowsize=2pt 2.5]{->}(2.8,.85)(3.2,.85)
\psline[linewidth=.5pt, arrowsize=2pt 2.5]{<-}(2.8,.65)(3.2,.65)
\endpspicture
\caption{\label{ed} Curves from the proof of Lemma \ref{bpinY}.}
\end{figure}

\begin{lemma}\label{bpinY}
Every bounding pair map belongs to $\Y(N)$.
\end{lemma}

\begin{proof} Let $T_{\gamma}T^{-1}_{\gamma'}$ be a bounding pair map, where $\gamma\cup \gamma'$ is the boundary of an orientable subsurface  $S$ of genus $j-1$ and the twists $T_\gamma$, $T_{\gamma'}$ are right with respect to some orientation of $S$. If $\gamma$ and $\gamma'$ are separating, then we are done by Lemma \ref{separating}, therefore we are assuming that $\gamma$ and $\gamma'$ are nonseparating.

If $N\backslash S$ is nonorientable, then there is a homeomorphism $h\colon N\to N$ which maps $(\gamma, \gamma')$ to
$(\beta_j, \beta'_j)$ (Figure \ref{betaj}), so
$hT_{\gamma}T^{-1}_{\gamma'}h^{-1}=(T_{\beta_j}T^{-1}_{\beta'_j})^{\pm 1}$ and it suffices to prove that $T_{\beta_j}T^{-1}_{\beta'_j}\in\Y(N)$. We have $RT_{\beta_j}R=T_{\beta'_j}$ and
$T_{\beta_j}T^{-1}_{\beta'_j}=T_{\beta_j}RT^{-1}_{\beta_j}R$, and since $R$ is in $\Y(N)$ by Lemma \ref{RinY}, thus $T_{\beta_j}RT^{-1}_{\beta_j}\in\Y(N)$ and $T_{\beta_j}T^{-1}_{\beta'_j}\in\Y(N)$.

Suppose that $N\backslash S$ is orientable. This can only happen if $g$ is even, say $g=2r$, and $N\backslash\gamma$ and $N\backslash\gamma'$ are orientable. We proceed by induction on $j$. If $j=1$ then $S$ is an annulus and 
$T_{\gamma}T^{-1}_{\gamma'}=1$. If $j=2$ then 
$T_{\gamma}T^{-1}_{\gamma'}$ is conjugate in $\M(N)$ to $(T_{\beta_r}T^{-1}_{\phi})^{\pm 1}$, where $\beta_r$ is the curve in Figure \ref{betaj} and $\phi$ is the curve in Figure \ref{ed}. As in the previous case, it suffices to show that $T_\phi=fT_{\beta_r}f^{-1}$ for some $f\in\Y(N)$. It is easy to check that we may take $f=U_2^{-1}T_\varepsilon=Y_2^{-1}T_{\alpha_2}^{-1}T_\varepsilon$, where $\varepsilon$ is the curve in Figure \ref{ed}. To see that $T_{\alpha_2}^{-1}T_\varepsilon\in\Y(N)$ notice that $\varepsilon$ and $\alpha_2$ bound a M\"obius strip with two holes, and by Lemma \ref{push2}, $T_{\alpha_2}^{-1}T_\varepsilon$ is in the image of the crosscap pushing map $\psi$, which is a subgroup of $\Y(N)$.

If $j>2$ then let $\gamma''$ be a separating simple closed curve in $S$ such that
$\gamma\cup\gamma''=\bdr{S'}$, $\gamma'\cup\gamma''=\bdr{S''}$ for subsurfaces
$S'$, $S''$ of genera less than $j-1$. We have 
$T_{\gamma}T^{-1}_{\gamma'}=(T_{\gamma}T^{-1}_{\gamma''})(T_{\gamma''}T^{-1}_{\gamma'})\in\Y(N)$
because the bounding pair maps on the right hand side are in $\Y(N)$ by the inductive hypothesis.
\end{proof}

\begin{figure}
\input{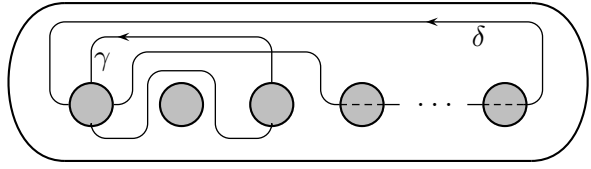}
\caption{\label{gd} Curves from the proof of Lemma \ref{y1}.}
\end{figure}

\begin{lemma}\label{y1}
Every crosscap slide is a product of crosscap slides conjugate in $\M(N)$ to $Y_1$.
\end{lemma}

\begin{proof} Let $Y=Y_{\mu,\alpha}$ be any crosscap slide and let $K$ and be a regular neighbourhood of $\alpha\cup\mu$, which is homeomorphic to the Klein bottle with a hole. Clearly there is a homeomorphism $h\colon K\to K_1$, where $K_1$ is a regular neighbourhood of $\alpha_1\cup\mu_1$, such that $h(\alpha)=\alpha_1$ and $h(\mu)=\mu_1$. If $N\backslash K$ is nonorientable or $g=2$, then $h$ may be extended to $h\colon N\to N$ and by (\ref{hyh}) we have 
$hY_{\mu,\alpha}h^{-1}=Y_{\mu_1,\alpha_1}$, hence $Y$ is conjugate to $Y_1$ in $\M(N)$. Suppose that $g>2$ and $N\backslash K$ is orientable. This forces the genus of $N$ to be even, $g=2r\ge 4$. Now $Y_{\mu,\alpha}$ is conjugate in $\M(N)$ to
$Y_{\mu_1,\beta_r}$, where $\beta_r$ is the curve in Figure \ref{betaj} for $j=r$ ($N\backslash\beta_r$ is orientable). 
Note that $Y_{\mu_1,\beta_r}$ is in the image of the crosscap pushing map
$\psi\colon\pi_1(N_{g-1},x_0)\to\M(N)$.
Consider the two-sided curves $\gamma$, $\delta$ in Figure \ref{gd}. Since 
$\psi$ is a homomorphism, sliding $\mu_1$ along $\beta_r$ is the same, up to isotopy, as sliding $\mu_1$ along $\alpha_1$ first, then along $\gamma$, then along $\delta$. Thus 
$Y_{\mu_1,\beta_r}=Y_{\mu_1,\delta}Y_{\mu_1,\gamma}Y_{\mu_1,\alpha_1}$, where each of the three crosscap slides on the right hand side is conjugate to $Y_1$ by the previous argument.  \end{proof}

\begin{theorem}\label{invol}
The group $\Y(N)$ is generated by involutions.
\end{theorem}
\begin{proof} Since $\Y(N)$ is a normal subgroup of $\M(N)$ and conjugate of an involution is also an involution, it suffices, by Lemma \ref{y1}, to show that $Y_1$ is a product of involutions in $\Y(N)$. The reflectional involution $R$ is in $\Y(N)$ by Lemma \ref{RinY}. Since $R$ maps $\mu_1$, $\alpha_1$ to
$\mu_1^{-1}$, $\alpha_1^{-1}$, we have 
$RY_1R=Y_1^{-1}$, thus $RY_1$ is an involution in $\Y(N)$ and $Y_1$ is a product of two involutions $Y_1=R(RY_1)$.\end{proof}

\section{Level $2$  mapping class groups}\label{s4}

In this section we assume that $S=S_{g,n}$ is an orientable surface of genus $g\ge 2$ having $n\in\{0,1\}$ boundary components. 

The action of $\M(S)$ on $H_1(S;\Z)$ preserves the algebraic intersection pairing, which is a nondegenerate alternating form, so we have a  representation of $\M(S)$ into the symplectic group, which is well known to be surjective and whose kernel $\calI(S)$ is known as the {\it Torelli group}. This is summarized by the exact sequence
\[1\longrightarrow\calI(S)\longrightarrow\M(S)\longrightarrow\Sp(2g,\Z)\longrightarrow 1.\] 
Wy are going to use the following classical theorem proved by Powell \cite{Powell} building on Birman's result \cite{Bir}.
\begin{theorem}[Powell \cite{Powell}]\label{BP}
The Torelli group $\calI(S)$ is generated by Dehn twists about separating curves and bounding pair maps.
\end{theorem}
\noindent This result was later improved by Johnson \cite{John1} who showed, that for $g\ge 3$, $\calI(S)$ is generated by finitely many bounding pair maps.

\medskip

For $m\ge 1$ let $\Sp(2g,\Z)[m]$ denote the {\it level $m$ congruence subgroup} of $\Sp(2g,\Z)$, that is, the subgroup of matrices that are equal to the identity modulo $m$. There is an exact sequence
\[1\longrightarrow\Sp(2g,\Z)[m]\longrightarrow\Sp(2g,\Z)\longrightarrow\Sp(2g,\Z_m)\longrightarrow 1,\]
where $\Z_m=\Z/m\Z$.
The pull-back of $\Sp(2g,\Z)[m]$ to $\M(S)$ is known as the {\it level $m$ subgroup} of $\M(S)$ and is denoted by $\Gamma_m(S)$. The group $\Gamma_m(S)$ can also be described as the group of isotopy classes of homeomorphisms that act trivially on $H_1(S;\Z_m)$. It fits into an exact sequence
\begin{equation}\label{ses3}
1\longrightarrow\calI(S)\longrightarrow\Gamma_m(S)\longrightarrow\Sp(2g,\Z)[m]\longrightarrow 1.
\end{equation}

It is well known that for $m \ge 3$, $\Sp(2g,\Z)[m]$ and $\Gamma_m(S)$ are torsion free. The case $m=2$ is exceptional, 
$\Gamma_2(S_{g,0})$ contains torsion elements, for example the hyperelliptic involution.

\medskip

For $h_1, h_2\in H_1(S;\Z)$ let $i(h_1,h_2)$ denote the algebraic intersection number. 
Let us fix a symplectic basis $\{a_1,\dots,a_g,b_1,\dots,b_g\}$ of $H_1(S;\Z)$, that is \[i(a_i,b_j)=\delta_{ij},\quad i(a_i,a_j)=0, \quad i(b_i,b_j)=0,\]
for $1\le i,j\le g$, where $\delta_{ij}$ is the Kronecker delta. We also fix a realisation of this basis by simple closed curves $\{\alpha_1,\dots,\alpha_g,\beta_1,\dots,\beta_g\}$, such that
$a_i=[\alpha_i]$, $b_i=[\beta_i]$ and
\[|\alpha_i\cap \beta_j|=\delta_{ij},\quad \alpha_i\cap \alpha_j=\emptyset, \quad \beta_i\cap \beta_j=\emptyset,\]
for $1\le i,j\le g$.
Recall that a homology class $h\in H_1(S;\Z)$ can be realised by a simple closed curve if and only if $h$ is primitive, which means that the greatest common divisor of the coefficients of $h$ in a symplectic basis equals 1.
For $f\in\M(S)$ we denote by $\overline{f}$ the   induced automorphism of $H_1(S;\Z)$ and, by abuse of notation, its matrix with respect to the basis  $\{a_1,\dots,a_g,b_1,\dots,b_g\}$. If $\gamma$ is a simple closed curve on $S$, then the right Dehn twist $T_\gamma$ acts on $H_1(S;\Z)$ by the transvection
\[\overline{T_\gamma}(h)=h+i(h,[\gamma])[\gamma].\]
From the above formula it is immediate that $T^2_\gamma\in\Gamma_2(S)$. The following lemma follows from \cite[Lemma 5]{John3}.

\begin{lemma}\label{mtwists1}
For $g\ge 2$ every element of $\Sp(2g,\Z)[2]$ is induced by a product of 
squares of Dehn twists about nonseparating curves on $S_{g,1}$.
\end{lemma}

From the exactness of sequence (\ref{ses3}), Theorem \ref{BP} and Lemma \ref{mtwists1} we have the following. 

\begin{cor}\label{gen1}
For $g\ge 2$ the group $\Gamma_2(S_{g,1})$ is generated by Dehn twists about separating curves, bounding pair maps and squares of Dehn twists about nonseparating curves.
\end{cor}

For the rest of this section assume that $S=S_g$ is closed and $\Sigma=S_{g-1,2}$.
Consider an embedding $i\colon\Sigma\hookrightarrow S$ such that $S\backslash i(\Sigma)$ is an annulus with core $\beta_g$. Such embedding is called capping in \cite{Put}. Since every homeomorphism of $i(\Sigma)$ equal to the identity on $i(\bdr{\Sigma})$ can be extended by the identity on $S\backslash i(\Sigma)$ to a homeomorphism of $S$, we have an induced homomorphism  $i_\ast\colon\M(\Sigma)\to\M(S)$. The image of $i_\ast$ is equal to the subgroup
$\Stab_{\M(S)}(\beta_g)$ consisting of the isotopy classes of homeomorphisms fixing the curve $\beta_g$ and preserving its orientation. Let $\delta_1$, $\delta_2$ be two simple closed curves in $\Sigma$ isotopic to the boundary components.  The kernel of $i_\ast$ is generated by the bounding pair map $T_{\delta_1}T^{-1}_{\delta_2}$. Summarising, we have an exact sequence
\[1\longrightarrow\lr{T_{\delta_1}T^{-1}_{\delta_2}}
\longrightarrow\M(\Sigma)\longrightarrow\Stab_{\M(S)}(\beta_g)\longrightarrow 1.\]
Following \cite{Put} we define $\calI(\Sigma)$ as
\[\calI(\Sigma)=i^{-1}_\ast(\calI(S)).\]
For a surface having more then one boundary component various definitions of a Torelli group are possible. The group $\calI(\Sigma)$ defined above is the same as 
$\calI(\Sigma,\{\{\delta_1,\delta_2\}\})$ in the notation of \cite{Put} and it agrees with the definition of the Torelli group of a surface with boundary given by Johnson \cite{John2}.
Note that $\calI(\Sigma)$ {\it is not} the kernel of the action of $\M(\Sigma)$ on $H_1(\Sigma;\Z)$. For example, $T_{\delta_i}$ for $i=1,2$ act trivially on the homology group, but they are not elements of $\calI(\Sigma)$. Nevertheless, $\calI(\Sigma)$ can also be defined intrinsically, without reference to $S$, by using relative homology, see \cite{John2,Put}.  
Note that $\gamma$ is a bounding curve in $\Sigma$ if and only if $i(\gamma)$ is separating in $S$, and $(\gamma, \gamma')$ is a bounding pair in $\Sigma$ if and only if $(i(\gamma), i(\gamma'))$ is a bounding pair in $S$. 

\begin{theorem}[\cite{Put}]\label{P}
The group $\calI(\Sigma)$ is generated by Dehn twists about 
bounding curves and bounding pair maps.
\end{theorem}

Let $\Stab_{\Sp(2g,\Z)}(b_g)$ be the subgroup of $\Sp(2g,\Z)$ consisting of the (matrices of) automorphisms of $H_1(S;\Z)$ which  fix the homology class $b_g=[\beta_g]$, and  
\[\Stab_{\Sp(2g,\Z)[m]}(b_g)=\Sp(2g,\Z)[m]\cap\Stab_{\Sp(2g,\Z)}(b_g).\]
Observe that we have an isomorphism
\[\Stab_{\Sp(2g,\Z)}(b_g)/\Stab_{\Sp(2g,\Z)[m]}(b_g)\cong\Stab_{\Sp(2g,\Z_m)}(\overline{b_g}),\]
where $\overline{b_g}$ is the homology class in $H_1(S;\Z_m)$ of the curve $\beta_g$.

\begin{lemma}\label{mtwist2}
For $m\le 2$ and $g\ge 3$ every element of $\Stab_{\Sp(2g,\Z)[m]}(b_g)$ is induced by a product of 
$m$-th powers of Dehn twists about nonseparating curves on $S$ disjoint from $\beta_g$.
\end{lemma}

\begin{proof} Let $N$ be a regular neighbourhood of $\alpha_g\cup\beta_g$ and $S'=S\backslash N$. Note that $S'$ is a copy of $S_{g-1,1}$  and we have a symplectic splitting \[H_1(S;\Z)=H_1(N;\Z)\oplus H_1(S';\Z),\]
where $H_1(N;\Z)=\lr{a_g,b_g}$.
Let $A\in \Stab_{\Sp(2g,\Z)[m]}(b_g)$ and consider $h=Aa_g$. We have
\[i(h,b_g)=i(Aa_g,Ab_g)=i(a_g,b_g)=1,\]
thus
\[h=a_g+m(xb_g+h'),\] where $x\in\Z$, $h'\in H_1(S';\Z)$. Let $h''$ be a primitive element of $H_1(S';\Z)$ such that $h'=yh''$ for $y\in\Z$ (if $h'=0$, then $y=0$ and $h''$ is arbitrary) and let $\gamma$ be any nonseparating simple closed curve in $S'$ such that $h''=[\gamma]$. 
Let $\delta$ be a simple closed curve disjoint from $\beta_g$ and such that
$[\delta]=[\gamma]-[\beta_g]$ (Figure \ref{bg}). We have
$\overline{T^{m(x+y)}_{\beta_g}T^{my}_\delta}(h)=a_g$. Since $\overline{T^{m(x+y)}_{\beta_g}T^{my}_\delta}A$ is equal to the identity on $\lr{a_g,b_g}$, thus it restricts to a symplectic automorphism of $H_1(S';\Z)$
and hence it is induced by a homeomorphism of $S'$. If $m=2$ then this homeomorphism may be taken to be a product of squares of Dehn twists about nonseparating curves on $S'$ by Lemma \ref{mtwists1} (here we use the assumption $g\ge 3$). Since $T_{\beta_g}$ can be seen as a twist about a curve isotopic to $\beta_g$ and disjoint from $\beta_g$,  $A$ is induced by  a product of $m$-th powers of Dehn twists about nonseparating curves on $S$ disjoint from $\beta_g$.\end{proof}

\begin{figure}
\input{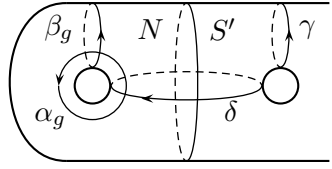}
\caption{\label{bg}Curves from the proof of Lemma \ref{mtwist2}.}
\end{figure}

We define level $2$ subgroup of $\M(\Sigma)$ as
\[\Gamma_2(\Sigma)=i^{-1}_\ast(\Gamma_2(S)).\]
Consider the map $\Gamma_2(\Sigma)\to\Stab_{\Sp(2g,\Z)[2]}(b_g)$ given by
$f\mapsto\overline{i_\ast(f)}$. By Lemma \ref{mtwist2} this map is surjective, because every simple closed curve in $S$ disjoint from $\beta_g$ is isotopic to a curve in $i(\Sigma)$. We have an exact sequence
\[1\longrightarrow\calI(\Sigma)\longrightarrow\Gamma_2(\Sigma)\longrightarrow
\Stab_{\Sp(2g,\Z)[2]}(b_g)\longrightarrow 1,\]
which together with Lemma \ref{mtwist2} and Theorem \ref{P} gives the following.

\begin{cor}\label{gen2}
Let $g\ge 3$ and $\Sigma=S_{g-1,2}$. The group $\Gamma_2(\Sigma)$ is generated by Dehn twists about bounding curves, bounding pair maps and squares of Dehn twists about nonseparating curves.
\end{cor}

\begin{prop}\label{quot}
The quotient group $\M(\Sigma)/\Gamma_2(\Sigma)$ is isomorphic to the stabiliser
$\Stab_{\Sp(2g,\Z_2)}(\overline{b_g})$.
\end{prop}

\begin{proof} 
The map $\Stab_{\M(S)}(\beta_g)\to\Stab_{\Sp(2g,\Z)}(b_g)$ is surjective
by Lemma \ref{mtwist2} for $m=1$.
We have
\[\Stab_{\Sp(2g,\Z)}(b_g)\cong\frac{\Stab_{\M(S)}(\beta_g)}{\calI(S)\cap\Stab_{\M(S)}(\beta_g)}
\cong\frac{\M(\Sigma)}{\calI(\Sigma)},\]
where the last isomorphism is induced by $i_\ast\colon\M(\Sigma)\to\Stab_{\M(S)}(\beta_g)$.
We have a series of isomorphisms
\[\frac{\M(\Sigma)}{\Gamma_2(\Sigma)}\cong
\frac{\M(\Sigma)/\calI(\Sigma)}{\Gamma_2(\Sigma)/\calI(\Sigma)}\cong
\frac{\Stab_{\Sp(2g,\Z)}(b_g)}{\Stab_{\Sp(2g,\Z)[2]}(b_g)}\cong
\Stab_{\Sp(2g,\Z_2)}(\overline{b_g}).
\]
\end{proof}

\section{The action of $\M(N)$ on $H_1(N;\Z_2)$.}\label{s5}
In this section we assume that $N=N_g$ is a closed nonorientable surface of genus $g\ge 2$. Recall that $H_1(N;\Z)$ is generated by $\{c_1,\dots,c_g\}$, where
$c_i=[\mu_i]$ for $1\le i\le g$ are the homology classes of the one-sided curves $\mu_i$ (Figure \ref{mui}). As a $\Z$-module it has the following presentation
\begin{equation}\label{pres}
H_1(N;\Z)=\lr{c_1,\dots,c_g\,|\,2(c_1+\cdots+c_g)=0}.
\end{equation}
The $\Z_2$-homology classes $\overline{c_i}$ of the curves
$\mu_i$ form a basis of $H_1(N;\Z_2)$, which is a vector space over $\Z_2$ of dimension $g$. The $\Z_2$-valued intersection pairing is a symmetric bilinear  form $\lr{\,{,}\,}$ on $H_1(N;\Z_2)$ satisfying $\lr{\overline{c_i},\overline{c_j}}=\delta_{ij}$ for $1\le i,j\le g$.
The following three theorems were proved by McCarthy and Pinkall \cite{McCP}.

\begin{theorem}[{\cite[Theorem 1]{McCP}}]\label{MCP1}
Every automorphism of $H_1(N;\Z)$ which preserves the $\Z_2$-valued intersection pairing is induced by a homeomorphism of $N$.
\end{theorem}

\begin{theorem}[{\cite[Theorem 2]{McCP}}]\label{MCP2}
Every automorphism of $H_1(N;\Z_2)$ which preserves the $\Z_2$-valued intersection pairing is induced by a homeomorphism of $N$.
\end{theorem}

\begin{theorem}[{\cite[Theorem 3]{McCP}}]\label{MCP3}
Every automorphism of $H_1(N;\Z)$ inducing a trivial automorphism of  $H_1(N;\Z_2)$ is induced by a homeomorphism of $N$ which is a product of crosscap slides.
\end{theorem}

By Theorem \ref{MCP2}, the natural map \[\M(N)\longrightarrow\Iso(H_1(N;\Z_2))\]
is surjective, where $\Iso(H_1(N;\Z_2))$ is the group of automorphisms of 
$H_1(N;\Z_2)$ which preserve $\lr{\,{,}\,}$.

Recall that a symmetric bilinear form $\Omega$ on a vector space $V$ over a field of characteristic 2 is called symplectic if it is nondegenerate  and $\Omega(v,v)=0$ for all $v\in V$. On a closed orientable surface $S_r$, the $\Z_2$-valued intersection pairing is a symplectic form on $H_1(S_r,\Z_2)$.
Let us fix a symplectic basis $\{\overline{a_1},\dots,\overline{a_r},\overline{b_1},\dots,\overline{b_r}\}$
of $H_1(S_r,\Z_2)$. We shell identify the group $\Sp(2r,\Z_2)$ with the group of symplectic automorphisms of $H_1(S_r,\Z_2)$. A proof of the following proposition can be found in \cite[Section 3]{KorkHom}.

\begin{prop}\label{order}
For $r\ge 1$, the group $\Iso(H_1(N_{2r+1},\Z_2))$ is isomorphic to the symplectic group $\Sp(2r,\Z_2)$, and its order is given by the formula
\[|\Iso(H_1(N_{2r+1},\Z_2))|=2^{r^2}\prod_{i=1}^r(2^{2i}-1),\]
the group $\Iso(H_1(N_{2r},\Z_2))$ is isomorphic to the stabiliser  $\Stab_{\Sp(2r,\Z_2)}(\overline{b_r})$, and its order is given by the formula
\[|\Iso(H_1(N_{2r},\Z_2))|=2^{r^2}\prod_{i=1}^{r-1}(2^{2i}-1).\]
\end{prop}

We define level 2 subgroup $\Gamma_2(N)$ of $\M(N)$ to be the group of isotopy classes of homeomorphisms which act trivially on $H_1(N;\Z_2)$. If fits into an exact sequence
\[1\longrightarrow\Gamma_2(N)\longrightarrow\M(N)\longrightarrow\Iso(H_1(N;\Z_2))
\longrightarrow 1.\]

The following theorem is an improvement of Theorem \ref{MCP3}. It is the main result of this paper. 

\begin{theorem}\label{main} A mapping class $f\in\M(N)$ acts trivially on $H_1(N;\Z_2)$ if and only if $f$ is a product of crosscap slides.
\end{theorem}

In other words $\Gamma_2(N)=\Y(N)$.

\begin{proof} The crosscap slide $Y_1=Y_{\mu_1,\alpha_1}$ acts on $H_1(N;\Z)$ by
\[\overline{Y_1}(c_1)=-c_1,\quad \overline{Y_1}(c_2)=2c_1+c_2,\quad
\overline{Y_1}(c_i)=c_i\textrm{\ for\ }3\le i\le g.\]
It follows that $Y_1$  acts trivially on $H_1(N;\Z_2)$ and hence, by Lemma \ref{y1}, $\Y(N)\subseteq\Gamma_2(N)$. For the equality it suffices to prove $[\M(N)\colon\Y(N)]\le[\M(N)\colon\Gamma_2(N)]$. The index $[\M(N)\colon\Gamma_2(N)]$ is equal to the order of $\Iso(H_1(N;\Z_2))$, which is given in Proposition \ref{order}. 

If $g=2$, then $\M(N)\cong\Z_2\oplus\Z_2$ , $\Gamma_2(N)=\Y(N)\cong\Z_2$ (see \cite{Lick1}).

If $g=3$, then $\M(N)\cong\mathrm{GL}(2,\Z)$ and it admits the presentation
\[\lr{a,b,y\,|\,aba=bab, y^2=1, yay=a^{-1}, yby=b^{-1}, (aba)^4=1},\]
where $a=T_{\alpha_1}$, $b=T_{\alpha_2}$, $y=Y_1$ (see \cite{BirChill}). By Lemma \ref{y1}, a presentation for the quotient group $\M(N)/\Y(N)$ may be obtained by adding the relation $y=1$ to the presentation of $\M(N)$. Thus
\[\M(N)/\Y(N)\cong\lr{a,b\,|\,a^2=b^2=(ab)^3=1}\cong \Sigma_3,\]
where $\Sigma_n$ is the symmetric group on $n$ letters. Since  $[\M(N)\colon\Gamma_2(N)]=6$ by Proposition \ref{order}, thus $\Y(N)=\Gamma_2(N)$.

If $g=4$, then we can use the presentation of $\M(N)$ given in \cite{Szep} to obtain the following presentation for $\M(N)/\Y(N)$.
\begin{align*}
&\M(N)/\Y(N)\cong
\lr{b\,|\,b^2=1}\times\\
&\lr{a_1,a_2,a_3\,|\,a_1^2=a_2^2=a_3^2=(a_1a_2)^3=(a_2a_3)^3=1, a_1a_3=a_3a_1},
\end{align*}
where $b=\widetilde{T_{\beta_2}}$, $a_i=\widetilde{T_{\alpha_i}}$ for $i=1,2,3$, where $\widetilde{f}$ is the coset of $f\in\M(N)$ in $\M(N)/\Y(N)$. It follows that
$\M(N)/\Y(N)\cong\Z_2\times\Sigma_4$, and since  $[\M(N)\colon\Gamma_2(N)]=48$ by Proposition \ref{order}, thus $\Y(N)=\Gamma_2(N)$.

Now we assume that $g=2r+1$ for $r\ge 2$. Let $\mu$ be a one-sided simple closed curve in $N$ such that $N\backslash\mu$ is orientable, and which is disjoint from the curves $\alpha_i$, $\beta_j$, $1\le i\le g-1$, $1\le j\le r$ in Figures \ref{alphai} and \ref{betaj}. Let $M$ be a regular neighbourhood of $\mu$, which is a M\"obius band, and $S=N\backslash M$.
Note that $S$ is a copy of $S_{r,1}$. Consider a homomorphism 
$j\colon\M(S)\to\M(N)/\Y(N)$ defined as $j=p\circ i_\ast$, where $i_\ast\colon\M(S)\to\M(N)$ is induced by the inclusion $S\hookrightarrow N$ and $p\colon\M(N)\to\M(N)/\Y(N)$ is the canonical projection. By Theorem \ref{Ch}, $\M(N)$ is generated by $Y_{g-1}$ and Dehn twists about curves in $S$. It follows that $j$ is surjective and hence
\[[\M(N)\colon\Y(N)]=[\M(S)\colon\ker{j}].\]
By Corollary \ref{gen1}, $\Gamma_2(S)$ is generated by Dehn twists about separating curves, bounding pair maps and squares of Dehn twists about nonseparating curves. It follows by Lemmas \ref{square}, \ref{separating} and \ref{bpinY} that
$i_\ast(\Gamma_2(S))\subseteq\Y(N)$, hence $\Gamma_2(S)\subseteq\ker{j}$.
We have
\[[\M(S)\colon\ker{j}]\le[\M(S)\colon\Gamma_2(S)]=|\Sp(2r,\Z_2)|.\]
Since $[\M(N)\colon\Gamma_2(N)]=|\Sp(2r,\Z_2)|$ by Proposition \ref{order}, thus $\Y(N)=\Gamma_2(N)$.

Finally we assume that $g=2r$ for $r\ge 3$. Let $A$ be regular neighbourhood of the two-sided curve $\beta_r$ and $\Sigma=N\backslash A$. Note that $\Sigma$ is a copy of $S_{r-1,2}$. The proof is analogous to the case of odd genus. The homomorphism $j\colon\M(\Sigma)\to\M(N)/\Y(N)$ defined as $j=p\circ i_\ast$, where $i_\ast\colon\M(\Sigma)\to\M(N)$ is induced by the inclusion $\Sigma\hookrightarrow N$, is surjective by Theorem \ref{Ch}. By Corollary \ref{mtwist2} and Lemmas \ref{square}, \ref{separating} and \ref{bpinY} we have
$\Gamma_2(\Sigma)\subseteq\ker{j}$ and thus
\[[\M(N)\colon\Y(N)]=[\M(\Sigma)\colon\ker{j}]\le[\M(\Sigma)\colon\Gamma_2(\Sigma)]=
|\Stab_{\Sp(2r,\Z_2)}(\overline{b_r})|,\]
the last equality following from Proposition \ref{quot}. By Proposition \ref{order} we have
\[[\M(N)\colon\Gamma_2(N)]=|\Stab_{\Sp(2r,\Z_2)}(\overline{b_r})|,\] and hence $Y(N)=\Gamma_2(N)$.\end{proof}

\begin{rem}
Theorem \ref{main} provides a characterisation of the elements of $\Y(N)$ in terms of the action on homology. Such characterisation is  known also for the twist subgroup. Namely, a mapping class $f\in\M(N)$ is a product of Dehn twists about two-sided curves if and only if the determinant of the induced map
$f_\ast\colon H_1(N;\R)\to H_1(N;\R)$ equals $1$ (see \cite[Corollary 6.3]{Stu2}).
\end{rem}

The following corollary is a consequence of Theorems \ref{invol} and \ref{main}.

\begin{cor}
The group $\Gamma_2(N)$ is generated by involutions.
\end{cor}

Consider the quotient $R_g=H_1(N_g,\Z)/\lr{c}$, where $c=c_1+\cdots+c_g$ is the unique homology class of order 2. It is immediate from the presentation \ref{pres} that $R_g$ is the free $\Z$-module with basis given by the images of $c_1,\dots,c_{g-1}$ in $R_g$. Every automorphism of $H_1(N_g,\Z)$ preserves $c$, and thus induces an automorphism of $R_g$. It was shown in \cite[Section 2]{McCP}, that the group of automorphisms of $H_1(N_g,\Z)$ which act trivially on $H_1(N_g,\Z_2)$ is isomorphic to the full group of automorphisms of $R_g$ which act trivially on $R_g\otimes\Z_2$. Consequently, we have a surjection
\[\Gamma_2(N_g)\longrightarrow\mathrm{GL}(g-1,\Z)[2],\]
where $\mathrm{GL}(n,\Z)[2]$ is the level 2 congruence subgroup of $\mathrm{GL}(n,\Z)$. The following corollary is probably well know to algebraists.

\begin{cor}
The group $\mathrm{GL}(n,\Z)[2]$ is generated by involutions.
\end{cor}

%

%
%

\end{document}